\documentclass[12pt]{amsart}
\usepackage{amssymb,amsmath,amsthm,latexsym}
\usepackage{mathrsfs}
\usepackage{a4wide}
\usepackage{eucal}
\usepackage[noadjust]{cite}
\usepackage[usenames,dvipsnames]{color}
\usepackage{stackengine}
\usepackage[utf8]{inputenc}
\usepackage[T1]{fontenc}
\usepackage{dsfont}
\theoremstyle{plain}

\DeclareFontFamily{U}{mathx}{\hyphenchar\font45}
\DeclareFontShape{U}{mathx}{m}{n}{
      <5> <6> <7> <8> <9> <10>
      <10.95> <12> <14.4> <17.28> <20.74> <24.88>
      mathx10
      }{}
\DeclareSymbolFont{mathx}{U}{mathx}{m}{n}
\DeclareFontSubstitution{U}{mathx}{m}{n}
\DeclareMathAccent{\widecheck}{0}{mathx}{"71}
\DeclareMathAccent{\wideparen}{0}{mathx}{"75}

\newtheorem{theorem}{Theorem}[section]
\newtheorem*{theorema*}{Theorem A}
\newtheorem*{theoremb*}{Theorem B}
\newtheorem{lemma}[theorem]{Lemma}

\theoremstyle{remark}

\newtheorem{remark}[theorem]{Remark}
\newtheorem*{remark*}{Remark}
\theoremstyle{definition}

\numberwithin{equation}{section}
\makeatother

\newcommand{\vertiii}[1]{{\left\vert\kern-0.25ex\left\vert\kern-0.25ex\left\vert #1 
    \right\vert\kern-0.25ex\right\vert\kern-0.25ex\right\vert}}

\newcounter{smallromans}

{\end{list}}

\def\dens{\operatorname{dens}}

\def\qe{\mathbb Q}

\newcounter{smallromansdash}

{\end{list}}

\newcounter{bigromans} 
  {\end{list}}

\begin{document}
\date{July 9, 2017}
\title{Restricting uniformly open surjections}

\author[T.~Kania]{Tomasz Kania}
\author[M.~Rmoutil]{Martin Rmoutil}
\address{Mathematics Institute,
University of Warwick,
Gibbet Hill Rd, 
Coventry, CV4 7AL, 
England}
\email{tomasz.marcin.kania@gmail.com, m.rmoutil@warwick.ac.uk}


\thanks{The authors acknowledge with thanks funding received from the European Research Council; ERC Grant Agreement No.~291497.}
\begin{abstract}We employ the theory of elementary submodels to improve a recent result by Aron, Jaramillo and Le Donne (\emph{Ann.~Acad.~Sci.~Fenn.~Math.}, to appear) concerning restricting uniformly open, continuous surjections to smaller subspaces where they remain surjective. To wit, suppose that $X$ and $Y$ are metric spaces and let $f\colon X\to Y$ be a~continuous surjection. If $X$ is complete and $f$ is uniformly open, then $X$ contains a~closed subspace $Z$ with the same density as $Y$ such that $f$ restricted to $Z$ is still uniformly open and surjective. Moreover, if $X$ is a Banach space, then $Z$ may be taken to be a closed linear subspace. A counterpart of this theorem for uniform spaces is also established. \end{abstract}
\maketitle

\section{Introduction}
Recently the problem of restricting surjective maps between metric spaces to \emph{smaller} subspaces where they remain being surjective attracted considerable attention due to a~renewed interest in possible abstract extensions of the Morse--Sard theorem. By the Axiom of Choice, every surjection $f\colon X\to Y$ admits a right inverse $g\colon Y\to X$, whence the range of $g$ is usually a smaller subspace of $X$ on which $f$ remains surjective. In the case where $X$ carries an extra structure, the range of a (highly non-constructively chosen) $g$ may be still quite large in a certain sense, though. Let us then make the problem more precise.\medskip

By a density of a topological space $Z$, we understand the smallest cardinality ${\rm dens}\, Z$ of a dense subset of $Z$. Suppose that $X$ and $Y$ are metric spaces and let $f\colon X\to Y$ be a~function. It is natural to ask under what circumstances should it be possible to find a subspace $Z$ of $X$ with ${\rm dens}\, Z = {\rm dens}\, Y$ such that $f$ restricted to $Z$ is still surjective as a function into $Y$. Aron, Jaramillo and Ransford (\cite{ajr}) proved that there exists a $C^\infty$-function from the non-separable Hilbert space $\ell_2(\mathbb{R}^2)$ onto $\mathbb{R}^2$, which fails to be surjective when restricted to any separable subset. On the positive side, Aron, Jaramillo and Le Donne (\cite[Theorem 1]{ajl}) proved that one may choose suitable $Z$ when the domain of $X$ is complete and $f$ is continuous and uniformly open. However it does not follow from their proof that the restriction may be taken uniformly open (or even open). (We state the definition of a uniformly open map in the subsequent section.)\medskip

\begin{theorema*}\label{T:GenAJL}
Suppose that $X$ and $Y$ are metric spaces and let $f\colon X\to Y$ be a continuous surjection. If $X$ is complete and $f$ is uniformly open, then $X$ contains a closed subspace $Z$ with ${\rm dens}\, Z = {\rm dens}\, Y$ such that $f$ restricted to $Z$ is uniformly open and surjective. Moreover, if $X$ is a Banach space, then $Z$ may be taken to be a closed linear subspace.  \end{theorema*}

Theorem A indeed strengthens the conclusion of \cite[Theorem~1]{ajl} by uniform openness of $f|_Z$ as even restrictions of uniformly open bounded bilinear maps on Banach spaces to closed subspaces need not be uniformly open. For example, the restriction of multiplication in the space of continuous functions on the Cantor set, which is uniformly open, to certain closed subalgebras is no longer so (see \cite[Section~3]{dk}).\smallskip

When this work was at the stage of completion, Le Donne has kindly communicated to us the claim that in joint work with Jaramillo they were able to relax the hypothesis of uniform openness of a~continuous surjection to mere openness by reducing the proof to the previously established uniformly open case. Our result is of different nature though. We show that a uniformly open continuous surjection $f\colon X\to Y$ may be restricted to a~uniformly open map on a~subspace $Z$ of $X$ with ${\rm dens}\, Z = {\rm dens}\, Y$ in such a way that the range of $f$ contains a~dense subset of $Y$. Since the domain of $f$ is complete and $f$ is uniformly open, the range of $f$ must be complete too, so it must be the whole $Y$. Thus, we shift our efforts from focusing on surjectivity of the restriction to showing mere uniform openness which would automatically imply surjectivity.\smallskip

Our seemingly overcomplicated proof is based on the method of elementary submodels, a~part of model theory. It is fair to say that the proof itself could be modelled on the proof of Theorem~\ref{T:elementary}, and the machinery from logic could be possibly avoided, however we have good reasons not to do this. The advantage of our approach is the ease with which we may impose further requirements on $Z$, if needed. For example, if $X$ carries an extra structure that can be expressed in terms of first-order logic (\emph{e.g.}, if $X$ is a normed space or a normed algebra), $f$ can be restricted to a closed substructure (a closed subspace or a~closed subalgebra, respectively) and remain uniformly open. In this case, avoiding using elementary submodels could bring the danger of being quickly lost in the obscurity of notation and other technical difficulties. Secondly, the problem itself appears to be tailor-made for the use of elementary submodels. Let us then take our proof as an opportunity for advertising the powerful method of elementary submodels; given its relative simplicity, the reader interested more in elementary submodels themselves than in our result, may regard the proof as a tutorial of the method.

\section{Preliminaries}

A map $f\colon X\to Y$ between metric spaces is \emph{uniformly open} when it has the property that for each $\varepsilon>0$ there is $\delta>0$ such that for each $x\in X$ one has $$B(f(x),\delta)\subseteq f(B(x,\varepsilon)).$$

It seems that the notion of uniform openness was first distilled by Michael \cite{michael}, however it had been employed already by Schauder (\cite[p.~6]{schauder}) \emph{en route} to the proof of his open-mapping theorem for complete metric vector spaces.

Let us record Schauder's lemma in its modern form (see, \emph{e.g.}, \cite[Lemma 3.9]{mv} for an~elementary proof).

\begin{lemma}[Schauder]\label{L:Schauder}
Let $X$ and $Y$ be metric spaces such that $X$ is complete. Suppose that $f\colon X\to Y$ is a continuous map. If for each $\varepsilon>0$ there exists $\delta>0$ such that for any $x\in X$ one has $$B(f(x),\delta)\subseteq \overline{f(B(x,\varepsilon))},$$ then $f$ is uniformly open.\end{lemma}

\begin{remark}\label{denseschauder}We shall need a formally stronger version of Schauder's lemma requiring that the postulated inclusion holds only for $x$ in a fixed, dense subset $D\subset X$. This is however sufficient. Indeed, let $\varepsilon > 0$ and $\delta>0$ corresponding to $\varepsilon /2$ be given. Take $x\in X$ and $x^\prime\in D$ such that $d(x,x^\prime)<\varepsilon/2$ and $d\big(f(x),f(x^\prime)\big)<\delta/2$. We then have $$ B(f(x), \delta / 2) \subseteq B(f(x^\prime), \delta) \subseteq  \overline{f(B(x^\prime, \varepsilon/2))} \subseteq  \overline{f(B(x, \varepsilon))}.$$
\end{remark}

We trust that the following basic property of uniformly open maps requires no explanation.
\begin{lemma}\label{L:surjective}
Let $X$ and $Y$ be metric spaces. Suppose that $X$ is complete. If $f\colon X\to Y$ is a uniformly open map with dense range, then $f$ is surjective.
\end{lemma}

Our method of proof uses elementary submodels; let us therefore describe some of the basic notions and facts we use. By \emph{formula} we shall always mean a formula in the language of ZFC. Let $N$ be a set and $\phi$ be a formula. The \emph{relativisation of $\phi$ to $N$} is the formula $\phi^N$, which is obtained form $\phi$ by replacing each quantifier of the form $\forall x$ by $\forall x\in N$ and each quantifier of the form $\exists x$ by $\exists x\in N$. Let $\varphi(x_1,\ldots, x_n)$ be a formula, where $x_1,\ldots, x_n$ are all the free variables of $\phi$. We say that \emph{$\phi$ is absolute for $N$} if 
\[ \forall a_1,\ldots a_n \in N \colon (\phi^N (a_1,\ldots,a_n)\leftrightarrow \phi(a_1,\ldots, a_n)).\]
We shall employ the following theorem (see, \emph{e.g.}, \cite[Chapter~IV, Theorem~7.8]{kunen}).
\begin{theorem}\label{T:elementary}
Let $\phi_1,\ldots,\phi_n$ be formulae and let $A$ be a set. Then there exists a set $M\supseteq A$ such that $\phi_1,\ldots,\phi_n$ are absolute for $M$ and $|M|\leqslant \max(\omega, |A|)$.
\end{theorem}
We refer the reader without background in logic to \cite[Chapter~24]{jw} for a leisurely exposition of elementary substructures and their applications outside set theory.
\section{Proof of Theorem~\ref{T:GenAJL}}
\begin{proof}[Proof of Theorem~\ref{T:GenAJL}]
It suffices to find a closed subset $Z\subset X$ with the same density as $Y$ so that $f|_Z$ is uniformly open and its range is dense in $Y$; indeed the surjectivity is then automatic by Lemma~\ref{L:surjective}. \smallskip

We may suppose that $Y$ is infinite as otherwise the statement is trivial. Take a set $D\subseteq X$ such that $f(D)$ is dense in $Y$ and  $|D|=\dens{Y}$, and set 
\[A=D\cup f(D) \cup \qe \cup\{f, X, d_X, Y, d_Y, <\},\] 
where $d_X$ and $d_Y$ are the metrics on $X$ and $Y$, respectively, and $<$ is the order relation in the real line (we consider the relation $<$ as a set, which we include into $A$ as an element, not a subset; similarly for $f, X, d_X, Y, d_Y$). Then $A$ has the same cardinality as $D$. Let $\phi$ be the following formula:
\[\forall \varepsilon \in \qe_+ \; \exists \delta\in \qe_+ \;\forall x\in X \; \forall y \;\exists z\in X \colon (d_Y(f(x),y)<\delta \rightarrow d_X(x,z)<\varepsilon \;\&\; y=f(z)).\]
Since $\phi$ is equivalent to $f$ being uniformly open, $\phi$ is true by the hypothesis of the theorem. By Theorem~\ref{T:elementary} we may find a set $M\supseteq A$ such that $\phi$ is absolute for $M$ and $|M|=\dens{Y}$. Set $Z=\overline{X\cap M}$; we \emph{claim} that $Z$ has the desired properties.\smallskip

The validity of $\phi$ and its absoluteness for $M$ imply that the following formula holds as well (note that $\qe_+\subseteq M$ and that all free variables appearing in this formula are also elements of $M$):
\begin{equation*}
\begin{aligned}
\forall \varepsilon \in \qe_+ \; \exists \delta\in \qe_+ \;\forall x\in X\cap M \; &\forall y\in M \;\exists z\in X\cap M \colon \\ &(d_Y(f(x),y)<\delta \rightarrow d_X(x,z)<\varepsilon \;\&\; y=f(z)).
\end{aligned}
\end{equation*}
This translates as follows---given $\varepsilon\in \qe_+$, there is $\delta\in \qe_+$ such that for each $x\in X\cap M$ we have 
\[ B(f(x),\delta)\cap M\subseteq f(B(x,\varepsilon)\cap M)\cap M,\]
whence
\[ B(f(x),\delta) \subseteq \overline{B(f(x),\delta)\cap M} \subseteq \overline{f(B(x,\varepsilon)\cap M)\cap M}\subseteq \overline{f(B_Z(x,\varepsilon))}.\]
Here the first inclusion follows from the fact that $f(D)\subseteq A\subseteq M$, so $M\cap Y$ is dense in $Y$; by $B_Z(x,\varepsilon)$ we mean simply $B(x,\varepsilon)\cap Z$, which makes the last inclusion trivial. Finally, by Lemma~\ref{L:Schauder} (which can be applied due to the continuity of $f$; see also Remark~\ref{denseschauder}), $f|_Z$ is uniformly open. As $f(D)$ is dense in $Y$, we are in a position to apply Lemma~\ref{L:surjective}, which concludes the proof.\smallskip

If $X$ is a Banach space we may have enlarged $A$ by the operation of addition in $X$ as well as by the operation $(\lambda, x)\mapsto \lambda x$, where $\lambda$ is a scalar and $x\in X$. In this case we may consider the formula $\psi$:
$$\forall x,y\in X\; \forall \lambda_1, \lambda_2\in \mathbb{Q} \; \exists z\in X \colon (\lambda_1 x + \lambda_2 y =z)$$
(or $\forall \lambda_1, \lambda_2\in \mathbb{Q}(i)$ in the case of complex scalars; we would then have included $\mathbb{Q}(i)$ in $A$ too). Theorem~\ref{T:elementary} applied to formulae $\varphi$, $\psi$ and the set $A$ yields a set $M$ for which $\varphi$, $\psi$ are absolute and so $\overline{M\cap X}$ is a~closed linear subspace of $X$ with the desired property.\end{proof}

\section{An extension to uniform spaces}
Generality of the employed method allows us to extend the result to uniformly open maps acting between uniform spaces. We refer the reader to James' book \cite{james} for the unexplained terminology concerning uniform spaces. \smallskip

Let $X$ and $Y$ be uniform spaces. A function $f\colon X\to Y$ is \emph{uniformly open} if for every entourage $\mathscr{D}$ of $X$ there is an entourage $\mathscr{E}$ of $Y$ such that $\mathscr{E}[f(x)] \subseteq f(\mathscr{D}[x])$ for each $x\in X$. Since every uniform space carries the canonical topological structure, we may talk about the density of a uniform space and continuity of maps defined therebetween.\smallskip

A~uniform space $X$ is \emph{super-complete} if the hyperspace $\exp X$ comprising all compact subsets of $X$ is complete when endowed with the Hausdorff uniformity; complete metric spaces with their natural uniformity are super-complete. Dektrajev (\cite{dektrajev}) proved that a~map $f\colon X\to Y$ between uniform spaces that has closed range is uniformly open as long as $X$ is super-complete and for every entourage $\mathscr{D}$ of $X$ there is an entourage $\mathscr{E}$ of $Y$ such that $\mathscr{E}[f(x)] \subseteq \overline{f(\mathscr{D}[x])}$ for every $x\in X$. Also, if $f\colon X\to Y$ is a uniformly open surjection, where $X$ is a complete uniform space, then $Y$ is complete too. Having prepared all the ingredients, by a completely analogous procedure, one may prove the following counterpart of Theorem A.

\begin{theoremb*}
Suppose that $X$ and $Y$ are uniform spaces and let $f\colon X\to Y$ be a continuous surjection. If $X$ is super-complete and $f$ is uniformly open, then $X$ contains a~closed subspace $Z$ with ${\rm dens}\, Z = {\rm dens}\, Y$ such that $f$ restricted to $Z$ is uniformly open and surjective. \end{theoremb*}

\subsection*{Acknowledgements} We are most grateful to Richard Aron for having explained to us the main result of \cite{ajl} during his visit to Warwick in April 2017.


\begin{thebibliography}{99}
\bibitem{ajl}  R.~Aron, J.~A.~Jaramillo and E.~Le Donne, Smooth surjections and surjective restrictions, to appear in \emph{Ann.~Acad.~Sci.~Fenn.~Math.}, \texttt{arXiv:1607.01725}.
\bibitem{ajr} R.~Aron, J.~A.~Jaramillo and T. Ransford, Smooth surjections without surjective restrictions, \emph{J. Geom. Anal.} \textbf{23} (2013) 2081--2090.
\bibitem{dektrajev}  I.~M.~Dektjarev, A closed graph theorem for ultracomplete spaces, \emph{Dokl. Akad. Nauk. Sov.}, \textbf{154} (1964), 771--773 (in Russian).
\bibitem{dk} Sz.~Draga and T.~Kania, When is multiplication in a Banach algebra open?, preprint (2017), 15~pp., \texttt{arXiv:1704.08608}.
\bibitem{james} I.~M.~James, \emph{Introduction to Uniform Spaces}, London Mathematical Society Lecture Note Series 144, Cambridge University Press, Cambridge, 1990.
\bibitem{jw} W.~Just and M.~Weese, \emph{Discovering Modern Set Theory. II: Set-Theoretic Tools for Every Mathematician}, Amer. Math. Soc., Providence, 1997.
\bibitem{kunen} K.~Kunen, \emph{Set Theory. An Introduction to Independence Proofs}, Studies Logic Found. Math., Vol. 102, North-Holland, Amsterdam (1980).
\bibitem{mv} R.~Meise and D.~Vogt, \emph{Introduction to Functional Analysis}, Clarendon Press, Oxford 1997.
\bibitem{michael} E.~Michael, Topologies on spaces of subsets, \emph{Trans. Amer. Math. Soc.}, \textbf{71} (1951), 152--182.
\bibitem{schauder}J.~Schauder, \"Uber die Umkehrung linearer, stetiger Funktionaloperatoren, \emph{Studia Math}.~\textbf{2} (1930), 1--6.

\end{thebibliography}
\end{document}